\documentclass[12pt]{amsart}
\usepackage{amssymb}
\usepackage{wasysym}

\newtheorem{thm}{Theorem}[section]
\newtheorem{lem}[thm]{Lemma}
\newtheorem{cor}[thm]{Corollary}
\newtheorem{prop}[thm]{Proposition}
\newtheorem{ex}[thm]{Example}

\newtheorem*{prob*}{Open problem}

\theoremstyle{definition}
\newtheorem{defi}[thm]{Definition}

\theoremstyle{remark}
\newtheorem{rem}[thm]{Remark}
\newtheorem*{rem*}{Remark}

\DeclareMathOperator{\id}{id}

\DeclareMathOperator{\rad}{rad}
\DeclareMathOperator{\Hom}{Hom}

\newcommand{\kringel}{\mathbin{\raise1pt\hbox{$\scriptstyle\circ$}}}
\newcommand{\pkt}{\mathbin{\raise0pt\hbox{$\scriptstyle\bullet$}}}
\newcommand{\ph}{\phantom{0000}}

\newcommand{\C}{\mathbb{C}}

\newcommand{\N}{\mathbb{N}}

\newcommand{\Z}{\mathbb{Z}}

\newcommand{\ad}{\mathop{\rm ad}}

\newcommand{\Der}{\mathop{\rm Der}}

\newcommand{\La}{\mathfrak{a}}
\newcommand{\Lb}{\mathfrak{b}}

\newcommand{\Ld}{\mathfrak{d}}
\newcommand{\Le}{\mathfrak{e}}
\newcommand{\Lf}{\mathfrak{f}}
\newcommand{\Lg}{\mathfrak{g}}
\newcommand{\Lh}{\mathfrak{h}}

\newcommand{\Ll}{\mathfrak{l}}
\newcommand{\Ln}{\mathfrak{n}}
\newcommand{\Lm}{\mathfrak{m}}
\newcommand{\Lp}{\mathfrak{p}}
\newcommand{\Lq}{\mathfrak{q}}
\newcommand{\Lr}{\mathfrak{r}}
\newcommand{\Ls}{\mathfrak{s}}
\newcommand{\Lt}{\mathfrak{t}}

\newcommand{\CL}{\mathcal{L}}

\newcommand{\al}{\alpha}
\newcommand{\be}{\beta}

\newcommand{\de}{\delta}

\newcommand{\la}{\lambda}

\newcommand{\ov}{\overline}

\newcommand{\ra}{\rightarrow}

\renewcommand{\phi}{\varphi}

\begin{document}

\title[Faithful modules]{Faithful Lie algebra modules and quotients of the universal
enveloping algebra}

\author[D. Burde]{Dietrich Burde}
\author[W. Moens]{Wolfgang Alexander Moens}
\address{Fakult\"at f\"ur Mathematik\\
Universit\"at Wien\\
  Nordbergstrasse 15\\
  1090 Wien \\
  Austria}
\email{dietrich.burde@univie.ac.at}
\address{Fakult\"at f\"ur Mathematik\\
Universit\"at Wien\\
  Nordbergstrasse 15\\
  1090 Wien \\
  Austria}
\email{wolfgang.moens@univie.ac.at}
\date{\today}

\thanks{The authors were supported by the FWF, Projekt P21683. The second
author was also supported by a Junior Research Fellowship of the ESI, Vienna.}

\begin{abstract}
We describe a new method to determine faithful representations of small 
dimension for a finite dimensional nilpotent Lie algebra. We give various
applications of this method. In particular we find a new upper bound on the
mi\-ni\-nmal dimension of a faithful module for the Lie algebras being counter examples
to a well known conjecture of J. Milnor. 
\end{abstract}

\maketitle

\section{Introduction}

Let $\Lg$ be a finite-dimensional complex Lie algebra. Denote by $\mu(\Lg)$ the
minimal dimension of a faithful $\Lg$-module. This is an invariant of $\Lg$, which
is finite by Ado's theorem. Indeed, Ado's theorem asserts that there exists a faithful linear 
representation of finite dimension for $\Lg$. There are many reasons why it is interesting to
study $\mu(\Lg)$, and to find good upper bounds for it. One important motivation comes
from questions on fundamental groups of complete affine manifolds and 
left-invariant affine structures on Lie groups. A famous problem of Milnor in this area
is related to the question whether or not $\mu(\Lg)\le \dim (\Lg)+1$ holds for
all solvable Lie algebras. For the history of this problem, and the counter examples to it
see \cite{MIL}, \cite{BU5} and the references given therein. \\
It is also interesting to find new proofs and refinements for Ado's theorem. We want to mention
the work of Neretin \cite{NER}, who gave a proof of Ado's theorem, which appears to be more
natural than the classical ones. This gives also a new insight into upper bounds for arbitrary
Lie algebras. \\
From a computational view, it is also very interesting to construct faithful representations 
of small degree for a given nilpotent Lie algebra $\Lg$. In \cite{BEG} we have given various 
methods for such constructions. In this paper we present another method using
quotients of the universal enveloping algebra, which has many applications
and gives even better results than the previous constructions. 
We obtain new upper bounds on the invariant $\mu(\Lg)$ for complex filiform
nilpotent Lie algebras $\Lg$. In particular, we find new upper bounds on $\mu(\Lg)$ for
the counter examples to Milnor's conjecture in dimension $10$.   \\[0.2cm]
The paper is organized as follows. After some basic properties we give estimates on $\mu(\Lg)$
in terms of $\dim (\Lg)$ according to the structure of the solvable radical of $\Lg$.
In the third section we describe the new construction of faithful modules by quotients of the
universal enveloping algebra. We decompose the Lie algebra $\Lg$ as a 
semidirect product $\Lg=\Ld\ltimes \Ln$, for some ideal $\Ln$ and  a subalgebra $\Ld\subseteq \Der(\Ln)$, 
and then constructing faithful $\Ld\ltimes \Ln$-submodules of $U(\Ln)$.
This is illustrated with two easy examples. \\
In the fourth section we give some applications of this construction. First we prove a bound
on $\mu(\Lg)$ for an arbitrary Lie algebra $\Lg$ in terms of the $\dim (\Lg/\Ln)$ and
$\dim(\Lr)$, where $\Ln$ denotes the nilradical of $\Lg$, and $\Lr$ the solvable radical.
Then we apply the construction to show that $\mu(\Lg)\le \dim (\Lg)$ for all $2$-step nilpotent
Lie algebras. Finally we apply the method to obtain new estimates on $\mu(\Lg)$ for
filiform Lie algebras $\Lg$, in particular for $\dim(\Lg)=10$. As for the counter examples
to Milnor's conjecture in dimension $10$, we give an example in $\ref{4.11}$.
It is quite difficult to see that this Lie algebra satisfies $\mu(\Lf)\ge 12$, so that it does
not admit an affine structure, see \cite{BU5}. On the other hand, it was known that  $\mu(\Lf)\le 22$.
Our new method gives $\mu(\Lf)\le 18$, which is up to now the best known upper bound.

%%%%%%%%%%%%%%%%%%%%%%%%%%%%%%%%%%%%%%%%%%%%%%%%%%%%%%%%%%%%%%%%%%%%%%%%%%%%%%

\section{Definitions and basic properties}

All Lie algebras are assumed to be complex and finite-dimensional, if not stated
otherwise.
Denote by $c$ the nilpotency class of a nilpotent Lie algebra.

\begin{defi}
Let $\Lg$ be a Lie algebra. We denote by $\mu(\Lg)$
the minimal dimension of a faithful $\Lg$-module, and by $\widetilde{\mu}(\Lg)$ the
minimal dimension of a faithful nilpotent $\Lg$-module.
\end{defi}

Note that $\widetilde{\mu}(\Lg)$ is only well-defined, if $\Lg$ is nilpotent.
On the other hand, every nilpotent Lie algebra admits a faithful nilpotent $\Lg$-module
of finite dimension \cite{BIR}. Recall the following lemma from \cite{BUM}.

\begin{lem}\label{mono}
Let $\Lh$ be a subalgebra of $\Lg$. Then $\mu(\Lh)\le \mu(\Lg)$. Furthermore,
if $\La$ and $\Lb$ are two Lie algebras, then $\mu(\La\oplus \Lb)\le \mu(\La)+\mu(\Lb)$.
\end{lem}

\begin{defi}
Denote by $\Lb_m$ the subalgebra of $\Lg\Ll_m(\C)$ consisiting of all 
upper-triangular matrices, by $\Ln_m=[\Lb_m,\Lb_m]$ the subalgebra
of all strictly upper-triangular matrices, and by $\Lt_m$ the subalgebra
of diagonal matrices.
\end{defi}

The following result is in principle well known. However, it appears in different
formulations, e.g., compare with Theorem $2.2$ in \cite{CAR}.

\begin{prop}\label{weight}
Let $\Ln$ be a nilpotent Lie algebra and $\rho\colon \Ln\ra \Lg\Ll(V)$
be a linear representation of $\Ln$ of degree $m$. Then there exists a basis
of $V$ such that $\rho$ can be written as the sum of representations
$\rho=\de+\nu$, such that
\begin{itemize}
\item[(1)] $\de (\Ln)\subseteq \Lt_m$ and $\nu(\Ln)\subseteq \Ln_m$.
\item[(2)] $\de ([\Ln,\Ln])=0$, and $\de$ and $\nu$ commute.
\item[(3)] $[\rho(x),\rho(y)] = [\nu(x),\nu(y)]$ for all $x,y\in \Ln$.
\end{itemize}
\end{prop}

\begin{proof}
By the weight space decomposition for modules of nilpotent Lie algebras we can
write 
\[
V=\bigoplus_{i=1}^s V^{\la_i}(\Ln),
\]
where $\la\in \Hom(\Ln,\C)$ are the different weights of $\rho$, and $V^{\la_i}(\Ln)$
are the weight spaces. In an appropriate basis of $V$ the operators $\rho(x)$ are given
by block matrices with blocks
\[
\begin{pmatrix}
\la_i(x) &  & \ast \\
         & \ddots & \\
0        &        &  \la_i(x)\\
\end{pmatrix}.
\]
Then let $\de(x)$ the diagonal part given by $\oplus_i \la_i(x)\id_{\mid V^{\la_i}}$, 
and put $\nu=\rho-\de$. Now it is easy to see that $\de$ and $\nu$ are representations. 
In fact, the $\la_i$ are characters, so that $\de ([\Ln,\Ln])=0$. Also, $\de$ commutes
with $\nu$, since it is a multiple of the identity on each block. This shows $(1)$ and $(2)$,
which in turn imply $(3)$.
\end{proof}

The next proposition gives an lower bound on $\mu(\Ln)$ in terms of the nilpotency class
of $\Ln$. As a special case we recover the well known estimate $n \le \mu(\Lf)$ for
a filiform Lie algebra $\Lf$ of dimension $n$.

\begin{prop}\label{2.5}
Let $\Ln$ be a nilpotent Lie algebra of class $c$ and dimension $n\ge 2$. Then 
we have $c + 1 \leq \mu(\Ln)$.
\end{prop}

\begin{proof}
If $\Ln$ is abelian, then $\mu(\Ln)\ge \lceil 2\sqrt{n-1}\rceil \ge 2=c+1$ by proposition
$2.4$ of \cite{BUM}.
Assume now that $\Ln$ is not abelian. Consider a faithful representation 
$\rho\colon \Ln\hookrightarrow \Lg\Ll(V)$ of degree $m$. 
Let $\rho = \de + \nu$ be a decomposition according to proposition $\ref{weight}$.
Then $[\rho(x),\rho(y)] = [ \nu(x),\nu(y)]$ for all 
$x,y\in \Ln$. Hence the non-trivial nilpotent Lie algebras $\rho(\Ln)$ and $\nu(\Ln)$ have the 
same nilpotency class $c$. 
Since the nilpotency class of $\Ln_m$ is $m-1$, and $\nu(\Ln)\subseteq \Lb$, it follows
$c\le m-1$. If we take $\rho$ to be of minimal degree, we obtain $c + 1 \leq \mu(\Ln)$.
\end{proof}

\begin{cor}
Let $\Lf$ be a filiform nilpotent Lie algebra of dimension $n$. Then
$n \le \mu(\Lf)$.
\end{cor}

There has been some interest lately in determining $\widetilde{\mu}(\Ln)$ for
nilpotent Lie algebras $\Ln$. We find that $\widetilde{\mu}(\Ln)$ coincides
with $\mu(\Ln)$ for a broad class of nilpotent Lie algebras.

\begin{lem}
Let $\Ln$ be a nilpotent Lie algebra satisfying $Z(\Ln) \subseteq [ \Ln , \Ln ]$. 
Consider a linear representation $\rho$ of $\Ln$ with above
decomposition $\rho = \de + \nu$. Then $\rho$ is faithful if and only if $\nu$ is.
\end{lem}

\begin{proof}
A representation of a nilpotent Lie algebra $\Ln$ is faithful if and only if 
the center $Z(\Ln)$ acts faithfully. Since $\rho(x)=\nu(x)$ for all $x,y\in [\Ln,\Ln]$,
and $Z(\Ln) \subseteq [ \Ln , \Ln ]$, $\rho$ and $\nu$ coincide on $Z(\Ln)$.
Hence the center acts faithfully by $\rho$ if and only if it acts faithfully by $\nu$.
\end{proof}

\begin{cor}
Let $\Ln$ be a nilpotent Lie algebra satisfying $Z(\Ln) \subseteq [\Ln ,\Ln ]$. 
Then $\mu(\Ln)=\widetilde{\mu}(\Ln)$.
\end{cor}

\begin{rem}
The condition $Z(\Ln) \subseteq [\Ln ,\Ln ]$ on nilpotent Lie algebras $\Ln$ is not
too restrictive. In fact, $\Ln$ always splits as $\C^{\ell}\oplus \Lm$ with
$Z(\Lm)\subseteq [\Lm,\Lm]$.
In particular, if the center is $1$-dimensional, or if $\Ln$ is indecomposable,
the condition is satisfied. This includes $\Ln$ being filiform nilpotent.
\end{rem}

%\begin{prop}\label{2.10}
%Let $\Lg$ be a solvable Lie algebra of solvability class $s\ge 1$.
%Then we have $2^{s - 2} \leq \mu(\Lg)$. 
%\end{prop}

%\begin{proof}
%Let $\rho\colon \Lg\ra \Lg\Ll (V)$ be a faithful representation of degree $m$.
%By Lie's theorem we may assume that $\rho(\Lg)\subseteq \Lb_m$. This implies,
%that $s$ is less or equal to the solvability class of $\Lb_m$, i.e.,
%$s\le \lceil \log_2 (m)\rceil +1$. We obtain $2^s\le 2^2m$, hence $2^{s-2}\le m$.
%Now choose $m$ to be minimal.
%\end{proof}

%\begin{rem}
%If $\Lg$ is nilpotent, of class $c$ and solvability class $s$, then 
%the estimates of proposition $\ref{2.5}$ and $\ref{2.10}$ are related as follows:
%\[
%2^{s-2}\le c+1 \le \mu(\Lg).
%\]
%\end{rem}

We are also interested in estimating $\mu(\Lg)$ in terms of $\dim (\Lg)$.
We present results which depend on the structure of the solvable radical of $\Lg$.
A first result is the following.

\begin{lem}
For any Lie algebra $\Lg$ we have $\sqrt{\dim (\Lg)}\le \mu(\Lg)$.
\end{lem}

\begin{proof}
Suppose that $\Lg$ can be embedded into some $\Lg\Ll_m(\C)$, then
\[
\dim (\Lg)\le \dim (\Lg\Ll_m(\C))=m^2. 
\]
In particular this holds for $m=\mu(\Lg)$.
\end{proof}

\begin{lem}\label{2.2}
Let $\Lg$ be represented as $\Lb\ltimes_{\de}\La$ for a Lie algebra $\Lb$ and an abelian Lie
algebra $\La$, such that the homomorphism $\de\colon \Lb\ra \Lg\Ll(\La)$ is faithful.
Then we have 
\[
\mu(\Lg)\le \dim (\La)+1.
\]
\end{lem}

\begin{proof}
Let $\dim(\La)=r$ and $\La\Lf\Lf (\La)=\Lg\Ll_r(\C)\ltimes_{\id} \C^r \subseteq \Lg\Ll_{r+1}(\C)$ 
be the Lie algebra of affine transformations of $\La=\C^r$. Define
\[
\phi\colon \Lb\ltimes_{\de}\La \ra \La\Lf\Lf (\La),\quad (b,a)\mapsto (\de (b),a).
\]
Then it is obvious that $\phi$ is faithful if and only if $\de$ is faithful.
Moreover the degree of the representation is $r+1$.
\end{proof}

Denote by $\rad(\Lg)$ the solvable radical of $\Lg$.

\begin{prop}\label{2.13}
Let $\Lg$ be a Lie algebra such that $\rad(\Lg)$ is abelian. Then we have
\[
\mu(\Lg)\le \dim (\Lg),
\]
and the only Lie algebras which satisfy equality are
the abelian Lie algebras of dimension $n\le 4$ and the Lie algebras
$\Le_8\oplus \cdots \oplus \Le_8$.
\end{prop}

\begin{proof}
The claim is clear for simple and abelian Lie algebras, see \cite{BUM}.
Since the $\mu$-invariant is subadditive, it also follows for reductive
Lie algebras. Now suppose that $\Lg$ is not reductive. Then we can even show
that $\mu(\Lg)\le \dim (\Lg)-2$. Let $\La=\rad(\Lg)$, and
$\Ls\ltimes_{\de}\La$ be a Levi decomposition,
where the homomorphism $\de\colon \Ls\ra \Lg\Ll(\La)$ is given by $\de(x)=\ad (x)_{\mid \La}$.
Since $\Ls$ is semisimple we can choose an ideal $\Ls'$ in $\Ls$ such that
$\Ls=\ker(\de)\oplus \Ls'$ and $\Lg= \ker(\de)\oplus (\Ls'\ltimes_{\de'} \La)$, where
$\de'=\de_{\mid \Ls'}$. Note that $\de'\colon \Ls'\ra \Lg\Ll (\La)$ is faithful. 
Now $\Ls'$ is non-trivial, since otherwise $\Lg=\ker(\de)\oplus \La$ would be reductive.
This implies $\dim (\Ls')\ge 3$ and $\dim (\ker(\de)) = \dim (\Ls)-\dim (\Ls')\le \dim (\Ls)-3$.
Since $\ker(\de)$ is semisimple, and by lemma $\ref{2.2}$ we obtain
\begin{align*}
\mu(\Lg) & \le \mu (\ker(\de))+\mu(\Ls'\ltimes_{\de'}\La)\\
         & \le \dim (\ker(\de))+\dim(\La)+1 \\
         & \le \dim(\Ls)-3+\dim(\La)+1 \\
         & = \dim (\Lg)-2.
\end{align*}
Finally we assume that $\mu(\Lg)=\dim (\Lg)$. By the above inequality, $\Lg$ needs to
be reductive. If $\Lg$ is simple, then only $\Lg=\Le_8$ satisfies the condition, see
\cite{BUM}. For a semisimple Lie algebra $\Ls=\Ls_1\oplus \cdots \oplus \Ls_{\ell}$
we have $\mu(\Ls)=\sum_i \mu(\Ls_i)$ and $\mu(\Ls_i)\le \dim (\Ls_i)$.
This implies that the only semisimple Lie algebras $\Ls$ satisfying $\mu(\Ls)=\dim(\Ls)$ 
are direct sums of $\Le_8$.
Also, the only abelian Lie algebras satisfying the condition are
the ones of dimension $n\le 4$. On the other hand, any reductive Lie algebra $\Lg$ satisfying
$\mu(\Lg)=\dim (\Lg)$ must be either semisimple or abelian:
if $\Lg=\Ls\oplus \C^{\ell+1}$ with $\ell \ge 0$ and a non-trivial
semisimple Lie algebra $\Ls$, then $\mu(\Ls\oplus\C)=\mu(\Ls)$, see \cite{BUM}, and
\begin{align*}
\mu(\Lg) & \le \mu(\Ls\oplus\C) +\mu(\C^{\ell}) \\
         & \le \mu(\Ls)+\ell \\
         & \le \dim (\Ls)+\ell \\
         & \le \dim (\Lg)-1.
\end{align*}
This is a contradiction, and we are done.
\end{proof}

Our next result is that $\mu(\Lg) \leq \dim(\Lg) + 1$ for any Lie algebra with $\rad(\Lg)$ abelian
or $2$-step nilpotent. We need the following two lemmas.

\begin{lem}\label{2.14}
Let $\Lg$ be a nilpotent Lie algebra and $D$ a derivation of $\Lg$ that induces an isomorphism 
on the center. Then $\mu(\Lg) \leq \dim(\Lg) + 1$.
\end{lem}

\begin{proof}
The center $Z(\Lg)$ is a nonzero characteristic ideal of $\Lg$, such that $D(Z(\Lg))\subseteq Z(\Lg)$.
Denote by $\Ld$ the $1$-dimensional Lie algebra generated by $D$, and form the split
extension $\Ld\ltimes \Lg$. By assumption this is a Lie algebra of dimension  $\dim(\Lg) + 1$ with
trivial center. Hence its adjoint representation $\ad \colon \Ld\ltimes \Lg \ra \Lg\Ll(\Ld\ltimes \Lg)$
is faithful. Together with the embedding $\Lg \hookrightarrow \Ld\ltimes \Lg$ we obtain
a faithful representation of $\Lg$ of degree $\dim(\Lg) + 1$. 
\end{proof}

\begin{lem} \label{2.15}
Let $\Lg$ be a Lie algebra with  Levi-decomposition
$\Lg = \Ls \ltimes \Lr$, such that $\Ls \leq \Der(\Lr)$. 
Suppose $D$ is a derivation of the radical $\Lr$. Then the map 
$\pi\colon \Ls \ltimes \Lr \ra \Ls \ltimes \Lr$ given by
$(X,t) \mapsto (0,D(t))$ is a derivation of $\Lg$ if and only if $[D,\Ls]=0$.
\end{lem}

\begin{proof} Consider any pair $a=(X,t)$ and $b=(Y,s)$ of elements in $\Lg$. We need to show 
that $\pi([a,b]) = [\pi(a),b] + [a,\pi(b)]$. The commutator of $a$ and $b$ is given by 
$ [(X,t),(Y,s)] = ([X,Y] , X(s) - Y(t) + [t, s])$ so that
\begin{align*}
\pi( [(X,t),(Y,s)] ) &= (0, D([t,s]) + (D \circ X) (s) - (D \circ Y) (t)) \\
     &= (0,[D(t),s] + [s,D(t)] + (D \circ X) (s) - (D \circ Y) (t)).
\end{align*}
We have $\pi((X,t)) = (0,D(t))$ and $\pi((Y,s)) = (0,D(s))$, hence
\begin{align*} [\pi((X,t)),(Y,s)] + [(X,t),\pi((Y,s))] & = [(0,D(t)),(Y,s)] + [(X,t),D(s)] \\ 
&=  (0, [D(t),s] + [t,D(s)] + (X \circ D)(s) \\
& - (Y \circ D)(t)). 
\end{align*}

We see that $\pi$ is a derivation of $\Lg$ if and only if these two expressions coincide for 
all $X,Y \in \Ls$ and all $s,t \in \Ln$. This is the case iff $[D,X](s) = 0$ for all $X \in \Ls$ 
and all $s \in \Ln$. This finishes the proof. 
\end{proof}

\begin{prop} Let $\Lg$ be a Lie algebra such that $\rad(\Lg)$ is nilpotent of class 
at most two. Then we have $\mu(\Lg) \leq \dim(\Lg) + 1$.
\end{prop}

\begin{proof}
Let $\Ls \ltimes \Ln$ be a Levi-decomposition for $\Lg$. If $\rad(\Lg)$ is abelian, the claim follows
from proposition $\ref{2.13}$.
Now assume that $\Ln$ is nilpotent of class two. As in the proof of proposition $\ref{2.13}$
we may assume that $\Ls$ acts faithfully on $\Ln$ and that $\Ls\subseteq \Der(\Ln)$. 
Now $\Ln_2 = [\Ln,\Ln]$ is an $\Ls$-submodule of $\Ln$, since $\Ls$ acts on $\Ln$ by derivations,
and $\Ln_2$ is invariant under these derivations, becuase it is a characteristic ideal.
Since $\Ls$ is semisimple, there exists an $\Ls$-invariant complement $\Ln_1$ to $[\Ln,\Ln]$. 
The $\Ls$-module decomposition $\Ln_1 + \Ln_2$ of $\Ln$ defines a linear transformation 
$D$ of $\Ln$ as follows: $D_{|\Ln_1} = \id_{\Ln_1}$ and $D_{|\Ln_2} = 2 \id_{\Ln_2}$. 
This is in fact a derivation of $\Ln$. Note that $D$ commutes with $\Ls$ in $\Der(\Ln)$. 
The derivation $D$ then extends to  a derivation $\pi$ of $\Lg = \Ls \ltimes \Ln$ by lemma
$\ref{2.15}$. Since $D$ is an isomorphism, $\pi_{\mid Z(\Lg)}$ is also an isomorphism.
By lemma $\ref{2.14}$, we may then conclude that $\mu(\Lg) \leq \dim(\Lg) + 1$. 
\end{proof}

\section{Quotients of the universal enveloping algebra}

\subsection{Order and length functions}
Let $\Ln$ be a nilpotent Lie algebra of dimension $n$ and class $c$. Consider a strictly 
descending filtration of $\Ln$ of the following form
\[
\Ln=\Ln^{[1]}\supset \Ln^{[2]} \supset \cdots \supset \Ln^{[C+1]}=0,
\]
where the $\Ln^{[i]}$ are subalgebras satisfying $[\Ln^{[i]},\Ln^{[j]}]\subseteq \Ln^{[i+j]}$
for all $1\le i,j\le C+1$. We say that the filtration is of {\it length} $C$, and we call
it an {\it adapted filtration}.
For example, such a filtration is given by the descending central series $\Ln^i$ for $\Ln$
of length $c$. To any such filtration associate a {\it order function}
\[
o:\Ln\ra \N\cup \{\infty\}, \quad x\mapsto \max_{t\in \N}\{x\in \Ln^{[t]}\}.
\]
If we let $\Ln^{[t]}=0$ for all $t\ge C+1$, then it makes sense to define $o(0)=\infty$.
It is easy to see that the order function $o$ satisfies the following two properties
\begin{align*}
o(x+y) & \ge \min \{o(x),o(y) \},\\
o([x,y]) & \ge o(x)+o(y)
\end{align*}
for all $x,y\in \Ln$. \\[0.2cm]
For a given subalgebra $\Lm$ of $\Ln$ satisfying $\Lm\supset \Ln^{[2]}$
we obtain an induced filtration
\[
\Lm\supset \Ln^{[2]} \supset \cdots \supset \Ln^{[C+1]}=0,
\]
and an associated order function. We extend the order function
to the universal enveloping algebra $U(\Ln)$ of $\Ln$ as follows.
Choose a basis $x_1,\ldots , x_n$ of $\Ln$ such that the first $n_1$ elements
span a complement of $\Ln^{[2]}$ in $\Ln$, the next $n_2$ elements span
a complement of $\Ln^{[3]}$ in $\Ln^{[2]}$, and so on. We identify the basis 
elements $x_i$ of $\Ln$ with the images $X_i$ in $U(\Ln)$ by the natural embedding.
The Poincar\'e-Birkhoff-Witt theorem
states that the monomials $X^{\al}=X_1^{\al_1}\cdots X_n^{\al_n}$ form a basis for $U(\Ln)$.
Now we set $o(X^{\al})=\sum_{j=1}^n \al_j o(X_j)$. For a linear combination
$W=\sum_{\al}c_{\al}X^{\al}$ we define $o(W)=\min_{\al}\{o(X^{\al})\mid c_{\al}\neq 0\}$.\\[0.2cm]
Furthermore we define a {\it length function} 
\[
\la\colon U(\Ln)\ra \N\cup \{ \infty \}
\]
by $\la(0)=\infty$, $\la(1)=0$ and
$\la(X^{\al})=\la (X_1^{\al_1}\cdots X_n^{\al_n})=\sum_{i=1}^n \al_i$.
Here $1$ denotes the unit element of $U(\Ln)$.
For a linear combination $W=\sum_{\al}c_{\al}X^{\al}$ we set 
$\la(W)=\min_{\al}\{\la(X^{\al})\mid c_{\al}\neq 0\}$. \\
The following result is well known for functions $o$ and $\la$ with respect
to the standard filtration of $\Ln$.
It easily generalizes to all adapted filtrations we have defined.

\begin{lem}\label{3.1}
For all $X,Y\in U(\Ln)$ we have the following inequalities:\\
\begin{itemize}
\item[(1)] $o(X+Y)\ge \min \{ o(X),o(Y)\}$.
\item[(2)] $o(XY)\ge o(X)+o(Y)$.
\item[(3)] $\la(X+Y)\ge \min \{ \la(X),\la(Y)\}$.
\item[(4)] $\la(X)\le o(X)$.
\end{itemize}
\end{lem}

Note that the elements of length $1$ are just the nonzero elements of $\Ln$.
Let
\[
V_t=\{ X\in U(\Ln)\mid o(X)\ge t \}.
\]
This is a $\Ln$-submodule of $U(\Ln)$, where the action is given by left-multiplication.
Furthermore we have $\Ln\cap V_t=\{0\}$ for all $t\ge C+1$.

\subsection{Actions on $U(\Ln)$}
The Lie algebra $\Ln$ acts naturally on $U(\Ln)$ by left multiplications.
We denote this action by $xY$, for $x\in \Ln$ and $Y\in U(\Ln)$.
We will show that semidirect products $\Ld\ltimes \Ln$ for subalgebras
$\Ld\le \Der(\Ln)$ also act naturally on $U(\Ln)$. First of all, $\Ld$ acts on 
$\Ln$ by derivations. Thus we already have an action of $\Ld$ on
the elements of length one in $U(\Ln)$. For $D\in \Der(\Ln)$ let $D(1)=0$ and
define recursively $D(XY)=D(X)Y+XD(Y)$ for all $X,Y\in U(\Ln)$. Then the
action of  $\Ld\ltimes \Ln$ on $U(\Ln)$ is given by
\[
(D,x).Y=D(Y)+xY
\]
for all $(D,x)\in  \Ld\ltimes \Ln$, and all $Y\in U(\Ln)$. 
This is well-defined, and we have the following useful lemma concerning
faithful quotients.

\begin{lem}\label{3.2}
Suppose that $W$ is a $\Ld\ltimes \Ln$-submodule of $U(\Ln)$ such that
$W\cap \Ln=0$. Then the quotient module $U(\Ln)/W$ is faithful.
\end{lem}

Consider a nilpotent Lie algebra $\Ln$ together with
the standard filtration given by the lower central series.
We have the following result.

\begin{prop}\label{proto}
Let $\Ln$ be a nilpotent Lie algebra of dimension $n$ and nilpotency class $c$. Let $\Ld$ be a subalgebra
of $\Der(\Ln)$ acting completely reducibly on $\Ln$. Then $V_{c+1}$ is a $\Ld\ltimes \Ln$-submodule 
of $U(\Ln)$ such that the quotient module $U(\Ln)/V_{c+1}$ is faithful of dimension 
at most $\frac{3}{\sqrt{n}}2^n$.
\end{prop}

\begin{proof}
Choose a basis for $\Ln$ associated to the standard filtration of $\Ln$ as in section $3.1$,
but with the additional requirement that each complement $C^{[i]}$ to $\Ln^{[i]}$
is also invariant under the action of $\Ld$, i.e., $D(C^{[i]})\subseteq C^{[i]}$ for all
$D \in \Ld$. This is possible since the $\Ln^{[i]}$ are characteristic ideals, hence invariant
under $\Ld$, so that they are submodules, which have a complementary submodule by the complete
reducibility.
Associate a PBW-basis for $U(\Ln)$ as before. Consider a basis element $x_j\in C^{[i]}$.
Then $o(x_j)=i$ and $o(D(x_j))\ge i$, since $D(x_j)$ is again in $C^{[i]}$, so has order $i$
or $\infty$. Hence it follows that $o(D(W))\ge o(W)$ for all $W\in U(\Ln)$. This means that
$V_{c+1}$ is a $\Ld$-submodule of $U(\Ln)$. Since we already know that $V_{c+1}$ is a $\Ln$-submodule,
it is a $\Ld\ltimes \Ln$-submodule of $U(\Ln)$. 
The quotient is faithful by lemma $\ref{3.2}$. Its dimension is bounded by $\frac{3}{\sqrt{n}}2^n$,
which was shown in \cite{BU7}, where it was considered just as an $\Ln$-module.
\end{proof}

\subsection{The construction of faithful quotients}
Let $\Ln$ be a nilpotent Lie algebra, together with some adapted filtration
$\Ln^{[t]}$ of length $C$, and a subalgebra $\Ld\le \Der(\Ln)$.

\begin{defi}\label{comp}
An ideal $J$ of $\Ln$ is called {\it compatible}, with respect to $\Ln^{[t]}$ and $\Ld$,
if it satisfies
\begin{itemize}
\item[(1)] $D(J)\subseteq J$ for all $D\in \Ld$,
\item[(2)] $J$ is abelian.
\item[(3)] $\Ln^{[t]}\subseteq J\subseteq \Ln^{[t+1]}$ for some $t\ge 0$.
\end{itemize}
\end{defi}

Denote by $\langle\langle J\rangle\rangle$ the linear subspace of $U(\Ln)$ generated by all
$Xy$ for $X\in U(\Ln)$ and $y\in J$. 
By assumption $J$ satisfies
\[
\Ln=\Ln^{[1]}\supset  \cdots \supset \Ln^{[t+1]}\supseteq J\supseteq \Ln^{[t]}\supset \cdots
\supset \Ln^{[C+1]}=0.
\]
For the rest of this section choose a basis $x_1,\ldots , x_n$ of $\Ln$ such that the 
first $n_1$ elements span a complement of $\Ln^{[2]}$ in $\Ln$, the next $n_2$ elements span
a complement of $\Ln^{[3]}$ in $\Ln^{[2]}$, and so on, including a basis
of a complement of $J$ in $\Ln^{[t+1]}$, and a complement of $\Ln^{[t]}$ in $J$.
A basis for $J$ is then of the form $x_m,\ldots ,x_n$ for some $m\ge 1$.
By the PWB-theorem we obtain standard monomials $X^{\al}$ in $U(\Ln)$ according to this basis.

\begin{lem}\label{3.4}
Let $J$ be a compatible ideal in $\Ln$. Then $\langle\langle J\rangle\rangle$ is 
the linear span of the standard monomials $X_1^{\al_1}\cdots X_n^{\al_n}$ with 
$(\al_m,\ldots ,\al_n)\neq (0,\ldots ,0)$. For any $W\in U(\Ln)$ and any $y\in J$ we have
$\la(Wy)\ge \la(W)+1$.
\end{lem}

\begin{proof}
First note that the monomials $X_1^{\al_1}\cdots X_n^{\al_n}$ with $(\al_m,\ldots ,\al_n)\neq 
(0,\ldots ,0)$ belong to $\langle\langle J\rangle\rangle$. They even span
$\langle\langle J\rangle\rangle$: assume that $T=X_{i_1}\cdots X_{i_{\ell}}$ is a standard
monomial of length $\ell$, and $x_k$ be a basis vector of $J$, i.e., $m\le k$. If $i_{\ell}\le k$ then
$Tx_k$ is one of our fixed standard monomials of length $\ell+1$, and obviously contained in 
$\langle\langle J\rangle\rangle$. Otherwise there exists a minimal $i_r$ such that
$i_{r-1}\le k<i_r$. Then, by definition of our basis for $\Ln$, all $X_{i_r}, \cdots , X_{i_{\ell}}$
are in $J$. Since $J$ is abelian, $X_{i_r} \cdots X_{i_{\ell}}x_k = x_kX_{i_r}\cdots X_{i_{\ell}}$.
Then we obtain $Tx_k=X_{i_1}\cdots X_{i_{r-1}}x_kX_{i_r} \cdots X_{i_{\ell}}$. This is a standard
monomial as above, contained in $\langle\langle J\rangle\rangle$, and of length $\ell+1$.
For an arbitrary element $W=\sum c_{\al}X^{\al}$ in $U(\Ln)$ we have, using $(3)$ of lemma $\ref{3.1}$,
\begin{align*}
\la(Wx_k) & = \la\bigl( \sum_{\al} c_{\al}X^{\al}x_k \bigr) \\
          & \ge \min_{\al} \{ \la(X^{\al}x_k)\} \\
          & \ge  \min_{\al} \{ \la(X^{\al})+1 \} \\
          & = \la(W)+1.
\end{align*}
Since the standard monomials $T=X^{\al}$ span $U(\Ln)$ as a vector space, the claim follows
by a similar computation.
\end{proof}

We define a subset 
\[
\CL_2=\langle W\in U(\Ln)\mid \la(W)\ge 2 \rangle
\]
of $U(\Ln)$.
Note that it is a vector space since a linear combination of elements of it is an
element again of length at least two. We have $\Ln\cap \CL_2=0$, since the nonzero
elements of $\Ln$ have length $1$.

\begin{lem}\label{3.5}
Let $J$ be a compatible ideal in $\Ln$, and $\Ld$ be a subalgebra of $\Der(\Ln)$. Then 
\[
W_J=\langle\langle J\rangle\rangle \cap \CL_2
\]
is a $\Ld\ltimes \Ln$-submodule of $U(\Ln)$, such that the quotient $U(\Ln)/W_J$ is
faithful.
\end{lem}

\begin{proof}
By the above remark, $W_J$ is a vector space. Let $x\in \Ln$, $W\in U(\Ln)$ and $x_k\in J$
such that $Wx_k \in W_J$. We want to show that $x(Wx_k)=(xW)x_k$ again is in $W_J$.
By definition it is in $\langle\langle J\rangle\rangle$. For the length we obtain,
using lemma $\ref{3.4}$, $\la((xW)x_k) \ge 1+\la(xW)\ge 2$. Hence $W_J$ is invariant
under the action of $\Ln$. Now we will show that $W_J$ is invariant under $\Ld$, so that
it is a $\Ld\ltimes \Ln$-submodule of $U(\Ln)$.
Let $D\in \Ld$ be a derivation. Then $D(Wx_k)=D(W)x_k+WD(x_k)$. Both terms on the RHS
are in $\langle\langle J\rangle\rangle$ by definition, and since $D(x_k)\in J$.
It remains to show that their length is at least $2$. Since by assumption $Wx_k\in W_J$, 
we have $\la(W)\ge 1$. This implies $\la(D(W))\ge 1$, and $\la(D(W)x_k)\ge \la(D(W))+1\ge 2$.
For the second term we obtain $\la(WD(x_k))\ge \la(W)+1\ge 2$. Since the sum of
two elements of length at least $2$ has lenght at least $2$, we obtain $D(Wx_k)\in W_J$.
Finally, we show that the quotient $U(\Ln)/W_J$ is faithful. By lemma $\ref{3.2}$ is suffices
to show that $\Ln\cap W_J=0$. This follows from $\Ln\cap W_J\subseteq \Ln\cap \CL_2=0$.
\end{proof}

We remark that the above quotient module will not yet be finite-dimensional in general.
We will achieve this by enlarging the submodule via $V_C$, where again
$V_t=\{X\in U(\Ln)\mid o(X)\ge t \}$, and $C$ denotes the length of the filtration attached to 
a compatible ideal $J$.

\begin{prop}\label{algo}
Let $J$ be a compatible ideal in $\Ln$. Suppose that $o(D(x))\ge o(x)+1$ for all $x\in \Ln$
and all $D\in \Ld$.Then 
\[
Z_J=\langle W_J, V_C\cap \CL_2\rangle
\]
is a $\Ld\ltimes \Ln$-submodule of $U(\Ln)$, such that the quotient $U(\Ln)/Z_J$ is
faithful and finite-dimensional.
\end{prop}

\begin{proof}
We first show that $\langle V_C\cap \CL_2\rangle$ is a $\Ld\ltimes \Ln$-submodule of 
$U(\Ln)$. The assumption also implies that $o(D(W))\ge o(W)+1$ for all $W\in U(\Ln)$.
Then for every $(D,x)$ in $\Ld\ltimes \Ln$ and every $W\in V_C\cap \CL_2$ we have
\begin{align*}
o((D,x).W) & = o(D(W)+xW) \\
           & \ge \min\{ o(D(W)),o(xW) \} \\
           & \ge o(W)+1 \\
           & \ge C+1.
\end{align*}
Hence $\langle V_C\cap \CL_2\rangle$ is mapped into $V_{C+1}$ under the action of 
$\Ld\ltimes \Ln$. But we have  $V_{C+1}\subseteq V_C\cap \CL_2$, because
$V_{C+1}\subseteq V_C$ and $V_{C+1}\subseteq \CL_2$. For the latter inclusion we note
that all elements of $V_{C+1}$ must have length at least $2$, since all elements of 
length at most $1$ are contained in $\Ln$, and $\Ln\cap V_{C+1}=0$. 
Hence all $(D,x).W$ are contained in $\langle V_C\cap \CL_2\rangle$. This implies
that $Z_J$ is a $\Ld\ltimes \Ln$-submodule, using lemma $\ref{3.5}$.
Since  $V_{C+1}\subseteq Z_J$ we have $\dim (U(\Ln)/Z_J)\le \dim(U(\Ln)/V_{C+1})$. Since  
the latter dimension is finite, we obtain that $U(\Ln)/Z_J$ is finite-dimensional.
Finally we show that the quotient module is faithful. Since $\Ln\cap Z_J\subseteq \Ln\cap \CL_2=0$
it follows from lemma $\ref{3.2}$.
\end{proof}

\subsection{Algorithmic construction}

We want to apply proposition $\ref{algo}$ to construct faithful modules of small dimension 
for a given nilpotent Lie algebras $\Lg$. 
The {\it input} is the Lie algebra $\Lg$ with a given basis, together with a 
decomposition $\Lg=\Ld\ltimes \Ln$,
for some ideal $\Ln$, a subalgebra $\Ld\subseteq \Der(\Ln)$, and choices of an admissible
filtration $\Ln^{[t]}$, a compatible ideal $J$, and so on, such that the assumptions of
the proposition are satisfied. The {\it output} will be a faithful $\Lg$-module of finite
dimension. How small this dimension is, will depend on clever choices of $\Ln$,$\Ld$, $J$,
$\Lg^{[t]}$, and so on. The algorithmic construction can be derived from the proof
of proposition $\ref{algo}$.
Let us illustrate this explicitly for the standard filiform Lie algebra $\Lg$ of dimension $4$,
with two different choices.
We choose a basis $x_1,\ldots x_4$ of $\Lg$ such that $[x_1,x_i]=x_{i+1}$ for $i=2,3$.

\begin{ex}
Write $\Lg=\Ld\ltimes \Ln$ with $\Ln=\langle x_1,x_3,x_4\rangle$ and $\Ld=\langle \ad(x_2)_{\mid \Ln}
\rangle$. Choose the filtration $\Ln=\Ln^{[1]}\supset \Ln^{[2]} \supset  \Ln^{[3]} \supset \Ln^{[4]}=0$
of length $C=3$ by $\Ln^{[2]}=\langle x_3,x_4\rangle$ and $\Ln^{[3]}=\langle x_4\rangle$.
Choose $J=\Ln^{[2]}$ as the compatible ideal. Then all conditions of the proposition are satisfied, 
and we obtain a faithful $\Lg$-module of dimension $5$.
\end{ex}

First note that we really have a filtration, $J$ is indeed a compatibe ideal, and the assumption
for the derivations in $\Ld$ is satisfied. 
Now the basis elements of order at most $3$ in $U(\Ln)$ are given as follows:
$1$ has order $0$; $X_1$ has order $1$; $X_3,X_1^2$ have order $2$, and
$X_1^3,X_1X_3,X_4$ have order $3$. Also, $1$ has length $0$, and $X_1,X_3,X_4$ have length $1$. 
Then we obtain 
\begin{align*} 
U(\Ln) & =\langle 1,X_1,X_3,X_1^2,X_1^3,X_1X_3,X_4\rangle +V_4, \\ 
\langle\langle J \rangle\rangle & =\langle X_3,X_1X_3,X_4\rangle + V_4',\\
W_J & =\langle X_1X_3\rangle+V_4''\\
Z_J & =\langle X_1X_3,X_1^3\rangle+V_4.
\end{align*}
where $V_4',V_4''$ are subspaces of $V_4$. 
Hence we obtain that
\[
U(\Ln)/Z_J=\langle \ov{1},\ov{X_1},\ov{X_3},\ov{X_1^2},\ov{X_4}\rangle
\]
where the bar denotes the cosets. This is a faithful $\Lg$-module of dimension $5$.
We can compute it explicitly, giving the action of the generators $x_1,x_2$ of $\Lg$.
\begin{align*}
x_1\cdot \ov{1}& =\ov{X_1},\; x_1\cdot \ov{X_1}=\ov{X_1^2},\; x_1\cdot \ov{X_3}=\ov{0},\;
x_1\cdot \ov{X_1^2}=\ov{0},\;  x_1\cdot \ov{X_4}=\ov{0}, \\
x_2\cdot \ov{1}& = \ov{0},\; x_2\cdot \ov{X_1}=[X_2,X_1]=- \ov{X_3},\; x_2\cdot \ov{X_3}=\ov{0},\; 
x_2\cdot \ov{X_1^2}=\ov{X_4},\; x_2\cdot \ov{X_4}=\ov{0}.
\end{align*}

Here we have 
\begin{align*}
x_2\cdot X_1^2& =[X_2,X_1^2] \\
 & = [X_2,X_1]X_1+X_1[X_2,X_1] \\
 & = -X_3X_1-X_1X_3 \\
 & = -[X_3,X_1]-2X_1X_3 \\
 & = X_4-2X_1X_3,
\end{align*}
so that $x_2\cdot \ov{X_1^2}=\ov{X_4}$.
Note that this $\Lg$-module has a submodule, generated by  $\ov{X_1^2}$ with a faithful
quotient of dimension $4$. Since $\mu(\Lg)=4$, the result is optimal. \\
In the second example we will directly obtain a faithful $4$-dimensional $\Lg$-module.
It will not be isomorphic to the above quotient module.

\begin{ex}\label{3.9}
Write $\Lg=\Ld\ltimes \Ln$ with $\Ln=\langle x_2,x_3,x_4\rangle$ and $\Ld=\langle \ad(x_1)_{\mid \Ln}
\rangle$. Choose the filtration $\Ln=\Ln^{[1]}\supset \Ln^{[2]} \supset  \Ln^{[3]} \supset \Ln^{[4]}=0$
of length $C=3$ by $\Ln^{[2]}=\langle x_3,x_4\rangle$ and $\Ln^{[3]}=\langle x_4\rangle$.
Choose $J=\Ln^{[1]}$ as the compatible ideal. Then all conditions of the proposition are satisfied, 
and we obtain a faithful $\Lg$-module of dimension $4$.
\end{ex}

Note that $J$ is an abelian ideal of codimension $1$ in $\Lg$. With $D=\ad(x_1)_{\mid \Ln}$ we have
$D(x_2)=x_3$ and $D(x_3)=x_4$. 
The elements of order at most $3$ in $U(\Ln)$ are given as follows: $1$ has order $0$; $X_2$ has 
order $1$; $X_3,X_2^2$ have order $2$, and $X_2^3,X_2X_3,X_4$ have order $3$. 
Then we obtain 
\begin{align*} 
U(\Ln) & =\langle 1,X_2,X_3,X_2^2,X_2^3,X_2X_3,X_4\rangle +V_4, \\ 
\langle\langle J \rangle\rangle & =\langle X_2,X_3,X_2^2,X_2^3,X_2X_3,X_4 \rangle + V_4',\\
W_J & =\langle X_2^2,X_2^3,X_2X_3\rangle+V_4''\\
Z_J & =\langle X_2^2,X_2^3,X_2X_3\rangle+V_4.
\end{align*}
where $V_4',V_4''$ are subspaces of $V_4$. Hence we obtain that
\[
U(\Ln)/Z_J=\langle \ov{1},\ov{X_2},\ov{X_3},\ov{X_4}\rangle.
\]
This is a faithful $\Lg$-module of dimension $4$. It is given by

\begin{align*}
x_1\cdot \ov{1}& =\ov{0},\; x_1\cdot \ov{X_2}=\ov{X_3},\; x_1\cdot \ov{X_3}=\ov{X_4},\;
x_1\cdot \ov{X_4}=\ov{0}, \\
x_2\cdot \ov{1}& = \ov{X_2},\; x_2\cdot \ov{X_2}=\ov{0},\; x_2\cdot \ov{X_3}=\ov{0},\; 
x_2\cdot \ov{X_4}=\ov{0}.
\end{align*}

\section{Applications}

\subsection{A general bound}

It is interesting to ask for good estimates on $\mu(\Lg)$ for arbitrary Lie
algebras. So far, general bounds have only been given for nilpotent Lie algebras.
For example, if $\Lg$ is nilpotent of dimension $r$ and of class $c$, then
$\mu(\Lg)\le \binom{r+c}{c}$, see \cite{GRA}. Independently of $c$ we have
$\mu(\Lg)\le \frac{3}{\sqrt{r}}2^r$, see \cite{BU7}.
There have been some attempts to find similar estimates for solvable
Lie algebras. We will present here such a bound for
arbitrary Lie algebras $\Lg$. Denote by  $\Ln$ the nilradical of $\Lg$, and by $\Lr$
its solvable radical. We may assume that $\Lr$ is non-trivial, because otherwise the 
adjoint representation is faithful. Hence let $\dim(\Lr)=r\ge 1$. We will show that 
$\mu(\Lg)\le \mu(\Lg/\Ln)+\frac{3}{\sqrt{r}}\cdot 2^r$. \\
We start with the following result of Neretin \cite{NER}, which we have slightly
reformulated for our purposes.

\begin{prop}\label{ner}
Let $\Lg$ be a complex Lie algebra with solvable radical $\Lr$ and Levi decomposition
$\Lg=\Ls\ltimes \Lr$. Let $\Lp$ be a reductive subalgebra of $\Lg$ and
$\Lm$ a nilpotent ideal satisfying the following properties:
\begin{itemize}
\item[(a)] $\Lp\cap \Lm=0$,
\item[(b)] $[\Lg,\Lr]\subseteq \Lm$ and $\Ls\subseteq \Lp$,
\item[(c)] $\Lp$ acts completely reducibly on $\Lm$. 
\end{itemize}
Then there exists a nilpotent Lie algebra $\Lh$ of dimension $\dim(\Lg)-\dim(\Lp)$
such that $\Lg$ embedds into a Lie algebra $(\Lp \oplus \C^{\ell})\ltimes \Lh$, with
$\ell=\dim(\Lg/(\Lp\ltimes\Lm))$, and the action of  $\Lp \oplus \C^{\ell}$ on $\Lh$
is completely reducible.
\end{prop}

We note the following corollary.

\begin{cor}
Let $\Lg$ be a complex Lie algebra with solvable radical $\Lr$ and nilradical $\Ln$. 
Then there exists a nilpotent Lie algebra $\Lh$ of dimension $\dim(\Lr)$ such that
$\Lg$ embedds into a Lie algebra $(\Lg/\Ln)\ltimes \Lh$, and the action of $\Lg/\Ln$
on $\Lh$ is completely reducible.
\end{cor}

\begin{proof}
In the notation of the above proposition write
$\Lg=\Ls\ltimes \Lr$ and choose $\Lp=\Ls$, and $\Lm=\Ln$.
Then the conditions $(a)-(c)$ are satisfied. Indeed, $\Ls\cap\Ln\subseteq \Ls\cap \Lr=0$.
Furthermore $[\Lg,\Lr]$ is a nilpotent ideal, hence is contained in $\Ln$. Finally,
$\Ls$ acts completely reducibly on $\Ln$, because $\Ls$ is semisimple. The result
follows.
\end{proof}

We obtain the following bound on $\mu(\Lg)$:

\begin{prop}
Let $\Lg$ be a complex Lie algebra with nilradical $\Ln$ and
solvable radical $\Lr$. Assume that $\dim(\Lr)=r\ge 1$. Then we have
\[
\mu(\Lg)\le \mu(\Lg/\Ln)+\frac{3}{\sqrt{r}}\cdot 2^r
\]
\end{prop}

\begin{proof}
We can embedd $\Lg$ into a Lie algebra $(\Lg/\Ln)\ltimes \Lh$ as in the corollary,
where $\Lh$ is a nilpotent Lie algebra of dimension $\dim(\Lr)$, and $\Lq=\Lg/\Ln$
is reductive. 
This means $\Lg \subseteq \Lq \ltimes \Lh$, and hence $\mu(\Lg)\le \mu(\Lq \ltimes \Lh)$ 
by lemma $\ref{mono}$. Now we want to apply proposition $\ref{proto}$ to $\Lq \ltimes \Lh$.  
For that we need that $\Lq$ is a subalgebra
of $\Der(\Lh)$, or equivalently, that $\Lq$ acts faithfully on $\Lh$. However,  
we may always decompose the reductive Lie algebra $\Lq$ as
$\Lq=\Lq_1\oplus \Lq_2$, where $\Lq_1$ commutes with $\Lh$, and $\Lq_2$ acts faithfully
and completely reducibly on $\Lh$. Again by lemma $\ref{mono}$, we obtain
$\mu(\Lq \ltimes \Lh)\le \mu(\Lq_1)+\mu(\Lq_2\ltimes \Lh)$.
We have $\mu(\Lq_1)\le \mu(\Lq)$ because of $\Lq_1\subseteq \Lq$. Furthermore we have
$\mu(\Lq)\le \dim (\Lq)$ by proposition $\ref{2.13}$. Now proposition $\ref{proto}$
can be applied to $\Lq_2\ltimes \Lh$, and we obtain
\begin{align*}
\mu(\Lg) & \le \mu(\Lq \ltimes \Lh) \\
         & \le \mu(\Lq_1)+\mu(\Lq_2\ltimes \Lh) \\
         & \le \dim (\Lq)+\frac{3}{\sqrt{r}}\cdot 2^r
\end{align*}
\end{proof}

\subsection{Two-step nilpotent Lie algebras}

It is well known that we have $\mu(\Lg)\le \dim(\Lg)+1$ for all
two-step nilpotent Lie algebras $\Lg$, see \cite{BU7}. It is not so easy to improve
this bound in general. Of course, for certain classes of two-step nilpotent Lie algebras
better bounds can be produced. We show the following result.

\begin{prop}
It holds $\mu(\Lg)\le \dim(\Lg)$ for all two-step nilpotent Lie algebras $\Lg$.
\end{prop} 

\begin{proof}
We can write $\Lg=\Lg_1\oplus \Lg_2$ with $Z(\Lg_2)\subseteq [\Lg_2,\Lg_2]$ and $\Lg_1$ 
abelian. Assume that we already know that $\mu(\Lg_2)\le \dim(\Lg_2)$. Then, by lemma 
$\ref{mono}$, it follows $\mu(\Lg)\le \mu(\Lg_1)+\mu(\Lg_2)\le 
\dim (\Lg)-\dim (\Lg_2)+\mu(\Lg_2)\le \dim (\Lg)$.
Hence we may assume that  $\Lg$ satisfies $Z(\Lg)\subseteq [\Lg,\Lg]$. Let $\dim(\Lg)=n$
and choose an ideal $\Ln\subseteq \Lg$ of codimension $1$ containing the commutator
of $\Lg$. Let $x_1,\ldots ,x_n$ be a basis of $\Lg$, such that $x_2,\ldots , x_n$ span $\Ln$.
Then $\Lg=\langle x_1\rangle \oplus \Ln$ as a vector space. Let $\Ld=\langle 
\ad (x_1)_{\mid \Ln}\rangle $, and we may write $\Lg=\Ld\ltimes \Ln$. Let 
$\Ln^{[1]}\supset \Ln^{[2]} \supset 0$ be the filtration of length $C=2$ given by $\Ln^{[1]}=\Ln$ 
and $\Ln^{[2]}=Z(\Lg)=[\Lg,\Lg]$. Recall here that $\Ln\supset [\Lg,\Lg]$.
Choose $J=Z(\Lg)$ as a compatible ideal. It satifies the conditions of definition $\ref{comp}$,
since it is invariant under all derivations of $\Ld$, and it is abelian.
Note that we have $D(\Ln^{[1]})\subseteq [\Lg,\Lg]=\Ln^{[2]}$ for all $D\in \Ld$, so that
$o(D(x))\ge o(x)+1$ for all $x\in \Ln$.
Now we can apply proposition $\ref{algo}$ with these choices. We obtain a faithful module
$U(\Ln)/Z_J=U(\Ln)/\CL_2$, which has dimension $n$, since it is spanned by the classes
of $1,x_2,\ldots ,x_n$.
\end{proof}

\subsection{Filiform nilpotent Lie algebras}

We wish to apply proposition $\ref{algo}$ to filiform nilpotent Lie algebras $\Lf$ of dimension
$n$ in order to improve the known upper bounds for $\mu(\Lf)$. 
Let $\Lf^1=\Lf$ and $\Lf^i=[\Lf,\Lf^{i-1}]$.
Let $\be(\Lf)$ be the maximal dimension of an abelian ideal of $\Lf$. It is well known
that $n/2\le \be(\Lf)\le n-1$.
Denote by $p_k(j)$ the number of partitions of $j$ in which each term does not exceed $k$.
Let $p_k(0)=1$ for all $k\ge 0$ and $p_0(j)=0$ for all $j\ge 1$.

\begin{prop}\label{4.5}
Let $\Lf$ be a filiform nilpotent Lie algebra of dimension $n$ having an abelian
ideal $J$ of dimension $1\le \be\le n-1$. Then we have $\mu(\Lf)\le f(n,\be)$, where
\[
f(n,\be)= \be +\sum_{j=0}^{n-2}p_{n-1-\be}(j).
\]
\end{prop}

\begin{proof}
Let $x_1,\ldots , x_n$ be an adapted basis of $\Lf$ in the sense of \cite{VER}. 
Then choose $\Ln=\langle x_2,\ldots ,x_n \rangle$ and
$\Ld=\langle \ad(x_1)_{\mid \Ln}\rangle$, so that $\Lf= \Ld\ltimes \Ln$. Define a filtration
$\Ln^{[1]}\supset \Ln^{[2]} \cdots \supset \Ln^{[C]}\supset 0$
of length $C=n-1$ by $\Ln^{[1]}=\Ln$ and $\Ln^{[i]}=\Lf^i$ for $i\ge 2$. 
We may write $J=\langle x_m,\ldots ,x_n\rangle$ with $m\ge 2$ and $n-m+1=\be$. 
It is easy to see that $J$ is a compatible ideal in the sense of definition $\ref{comp}$. 
Furthermore we have
$o(D(x))\ge o(x)+1$ for all $x\in \Ln$ and all $D\in \Ld$. Now we can apply proposition
$\ref{algo}$. We obtain a faithful module $U(\Ln)/Z_J$. We will show that its dimension is
$\be +\sum_{j=0}^{n-2}p_{n-1-\be}(j)$. It is generated by the classes 
\[
\{\ov{X_m},\ldots ,\ov{X_n}\}\;
\cup
\; \{\ov{X^{\al}}=\ov{X_2^{\al_2}\cdots X_{m-1}^{\al_{m-1}}}\mid o(X^{\al})\le n-2 \}.
\]
There are $\be$ monomials in the first set. The cardinality of the second set is given by
\begin{eqnarray*}
 & & \# \{ (\al_2,\ldots ,\al_{m-1})\in \Z_{\ge 0}^{m-2} \mid 1\cdot \al_2+2\cdot \al_3+\cdots 
+ (m-2)\cdot \al_{m-1}\le n-2 \}\\
 & = & \sum_{j=0}^{n-2} \# \{ (\al_2,\ldots ,\al_{m-1}) \mid  1\cdot \al_2+2\cdot \al_3+\cdots 
+ (m-2)\cdot \al_{m-1}=j \}\\ 
 & = & \sum_{j=0}^{n-2} p_{m-2}(j).
\end{eqnarray*}
Since $m-2=n-1-\be$ we obtain the required dimension.
\end{proof}

Note that for $\be=1$ we obtain the bound from  \cite{BU7}:
\[
\mu(\Lf)\le f(n,1)=1+\sum_{j=0}^{n-2}p(j)<1+e^{\sqrt{2\pi(n-1)/3}}.
\]
Here $p(j)$ denotes the unrestricted partition function, and  $p(0)=1$.
The following result shows that our bound from the above
proposition yields an improvement.

\begin{prop}
Let $n\ge 3$. Then $f(n,\be)$ is monotonic in $\be$, i.e., it holds
\[
f(n,n-1)\le f(n,n-2)\le \cdots \le f(n,2)=f(n,1),
\]
with equality for $\be=1$ and $\be=2$.
\end{prop}

The proof is easy, and we leave it to the reader. We can also determine
$f(n,\be)$ explicitly for large $\be$:

\begin{prop}
Let $n\ge 4$. Then it holds
\begin{align*}
f(n,n-1) & = n, \\
f(n,n-2) & = 2n-3, \\
f(n,n-3) & = \frac{n^2+3n-12+2\lfloor n/2\rfloor}{4}.
\end{align*}
\end{prop}

If $\be=n-1$, then $\be=\be(\Lf)$, and $\Lf$ is the standard graded filiform
Lie algebra. Then the bound $\mu(\Lf)\le f(n,n-1)=n$ is optimal, since we already know that 
$\mu(\Lf)=n$ in this case. See also example $\ref{3.9}$ for the case $n=4$.

\begin{rem}
It is also easy to show that 
\[
f(n,\be)\le \be+\frac{(2n-\be-3)^{n-\be-1}}{(n-\be-1)!}
\]
for all $n\ge 3$ and all $1\le \be \le n-1$.
\end{rem}

%\begin{ex}
%If $\be=n-3$ then we obtain
%\[
%\mu(\Lf)\le n-3+\sum_{j=0}^{n-2}p_2(j)=\begin{cases} \frac{n^2+4n-12}{4}, \quad n\equiv 0 
%\mod 2\\[0.2cm]
%\frac{n^2+4n-13}{4}, \quad n\equiv  1 \mod 2  \end{cases}
%\]
%\end{ex}

We can also derive a bound on $\mu(\Lf)$ which only depends on $n$.
For this we take the smallest possible $\be=\be(\Lf)$ in terms of $n$, which is 
given by  $\be=\lceil n/2 \rceil$. 
Then $n-1-\be=\lfloor n/2 \rfloor -1$, and we obtain the following result:

\begin{cor}
Let $\Lf$ be a filiform nilpotent Lie algebra of dimension $n\ge 3$. Then
\[
\mu(\Lf)\le n-1 +\sum_{j=0}^{n-2}p_{\left\lfloor \frac{n}{2} 
\right\rfloor -1}(j).
\]
\end{cor}

\newpage

\subsection{Filiform Lie algebras of dimension 10}
We may represent all complex filiform Lie algebras of dimension $10$ with respect to
an adapted basis $(x_1,\ldots ,x_{10})$ as a family of Lie algebras $\Lf=\Lf(\al_1,\ldots ,\al_{13})$,
with $13$ parameters satisfying the following polynomial equations:
\begin{align*}
\al_{11}(2\al_1+\al_7)-3\al_7^2 & = 0, \\
\al_{13}(2\al_1-\al_7-\al_{11}) & = 0, \\
\al_{13}(2\al_3+\al_9)-\al_{12}(2\al_1+\al_7) & = 3\al_{11}(\al_2+\al_8)-7\al_7\al_8.
\end{align*}
We call the parameters {\it admissible}, if they define a Lie algebra, i.e., if they satisfy 
these equations. Note that we obtain other equations as consequences, such as
\[
\al_{13}(\al_1^2-\al_7^2)=0.
\]
The explicit Lie brackets are given as follows:
\begin{align*}
[x_1,x_i] & =x_{i+1}, \; 2\le i\le 9 \\
[x_2,x_3] & = \al_1x_5+\al_2x_6+\al_3x_7 +\al_4x_8+\al_5x_9+\al_6x_{10}\\
[x_2,x_4] & = \al_1x_6+\al_2x_7+\al_3x_8+\al_4x_9 +\al_5x_{10}\\
[x_2,x_5] & = (\al_1-\al_7)x_7+ (\al_2-\al_8)x_8+
(\al_3-\al_9)x_9 +(\al_4-\al_{10})x_{10}\\
[x_2,x_6] & = (\al_1-2\al_7)x_8+ (\al_2-2\al_8)x_9 +(\al_3-2\al_9)x_{10}\\
[x_2,x_7] & = (\al_1-3\al_7+\al_{11})x_9 +(\al_2-3\al_8+\al_{12})x_{10}\\
[x_2,x_8] & = (\al_1 - 4\al_7 + 3\al_{11})x_{10}\\
[x_2,x_9] & = -\al_{13}x_{10}\\
[x_3,x_4] & = \al_7x_7+\al_8x_8+\al_9x_9+\al_{10}x_{10} \\
[x_3,x_5] & = \al_7x_8+\al_8x_9+\al_9x_{10} \\
[x_3,x_6] & = (\al_7-\al_{11})x_9 +(\al_8-\al_{12})x_{10}\\
[x_3,x_7] & = (\al_7-2\al_{11})x_{10}\\
[x_3,x_8] & = \al_{13}x_{10}\\
[x_4,x_5] & = \al_{11}x_9 + \al_{12}x_{10}\\
[x_4,x_6] & = \al_{11}x_{10}\\
[x_4,x_7] & = -\al_{13}x_{10}\\
[x_5,x_6] & = \al_{13}x_{10}
\end{align*}

We want to determine as good as possible upper bounds on $\mu(\Lf)$, for all Lie algebras
$\Lf=\Lf(\al_1,\ldots ,\al_{13})$. The results will depend on the parameters,
and we have to introduce a case distinction. For each case we choose a particular
construction which yields a faithful $\Lf$-module $V$ of some dimension $10\le \dim (V)\le 18$.
This improves the known bound $10\le \mu(\Lf)\le 22$ from \cite{BU5} for such Lie algebras.
We can also construct a faithful $\Lf$-module $V=V(\al_1,\ldots ,\al_{13})$, which
does not depend on a case distinction for the parameters. In other words, such a module
gives an upper bound on  $\mu(\Lf)$ for all admissible parameters at the same time.
We call such a module a {\it general} $\Lf$-module. We will give such a module explicitly.

\begin{prop}\label{4.10}
There is a general faithful $\Lf$-module $V_{58}=V_{58}(\al_1,\ldots ,\al_{13})$ of 
dimension $58$. 
\end{prop}

\begin{proof}
The faithful $\Lf$-module $V_{58}$ is obtained by proposition $\ref{4.5}$ as follows. 
Take $J=\langle x_6,\ldots ,x_{10}\rangle$ as compatible 
ideal. This means $\be=5$ and the construction yields a module with a basis consisting of 
$f(10,5)=58$ monomials. The computation of $f(10,5)$ uses $(p_4(0),\cdots ,p_4(8))=(1,1,2,3,5,6,9,11,15)$.
The basis consists of the following standard mononials, writing $x_i$ for $\ov{X_i}$.
\[
\begin{array}{c|c}
\text{order} & \text{monomials} \\
\hline 
0 & 1,\\
1 & x_2,\\
2 & x_3, x_2^2,\\
3 & x_4, x_2x_3, x_2^3,\\
4 & x_5, x_2x_4, x_3^2, x_2^2x_3, x_2^4,\\
5 & x_6, x_3x_4, x_2x_5, x_2^2x_4, x_2x_3^2x_3,x_2^3x_3,x_2^5,\\
6 & x_7, x_4^2, x_3x_5, x_2x_3x_4, x_3^3, x_2^2x_5,x_2^3x_4, x_2^2x_3^2, x_2^4x_3,x_2^6,\\
7 & x_8, x_4x_5, x_2x_4^2, x_3^2x_4, x_2x_3x_5,x_2^2x_3x_4, x_2x_3^3, x_2^3x_5,x_2^4x_4, x_2^3x_3^2,
  x_2^5x_3,x_2^7,\\
8& x_9, x_3x_4^2, x_5^2, x_2x_4x_5, x_3^2x_5,x_2^2x_4^2, x_2x_3^2x_4, x_3^4,x_2^2x_3x_5, x_2^3x_3x_4, \\
& x_2^2x_3^3, x_2^4x_5,x_2^5x_4, x_2^4x_3^2,x_2^6x_3,x_2^8,\\
9 & x_{10}. 
\end{array}
\]
Denote this basis by $v_1,\ldots ,v_{58}$, ordered lexicographically.
Note that $v_{58}=x_{10}$ generates the center of $\Lf$.
The module is determined by the action of the generators $x_1$ and $x_2$ of
the Lie algebra $\Lf=\Lf(\al_1,\ldots ,\al_{13})$. It is given by
\begin{align*}
x_1.v_1 & = 0,\\
x_1.v_2 & = v_3,\\
x_1.v_3 & = v_5,\\
x_1.v_4 & = 2v_6-\al_1v_8-\al_2v_{13}-\al_3v_{20}-\al_4v_{30}-\al_5v_{42}-\al_6v_{58},\\
x_1.v_5 & = v_8,\\
x_1.v_6 & = v_9+v_{10},\\
x_1.v_7 & = 3v_{11}-3\al_1v_{15}+\al_1(\al_1-\al_7)v_{20}+(2\al_1\al_2-2\al_2\al_7-\al_1\al_8)v_{30}\\
        & +(2\al_1\al_3 - \al_1\al_9 + \al_{11}\al_3 + \al_2^2 - 2\al_2\al_8 - 3\al_3\al_7)v_{42}
+ ( 2\al_1\al_4 - \al_1\al_{10}\\
        &  + 3\al_{11}\al_4 + \al_{12}\al_3 - \al_{13}\al_5 
           + 2\al_2\al_3 - 2\al_2\al_9 - 3\al_3\al_8 - 4\al_4\al_7)v_{58},
\end{align*}

\begin{align*}
x_1.v_8 & = v_{13},\\
x_1.v_9 & = v_{14} + v_{15},\\
x_1.v_{10} & = 2v_{14} - \al_7v_{20} - \al_8v_{30} - \al_9v_{42} - \al_{10}v_{58},\\
x_1.v_{11} & = v_{16} + 2v_{17} - \al_1v_{22} + \al_1\al_7v_{30} + (\al_1\al_8 - \al_{11}\al_2 
+ \al_2\al_7)v_{42}\\
          & + (\al_1\al_9 - 2\al_{11}\al_3 - \al_{12}\al_2 + \al_{13}\al_4 + \al_2\al_8 + \al_3\al_7)v_{58},\\
x_1.v_{12} & = 4v_{18} - 6\al_1v_{25} + \al_1(4\al_1\al_7 - \al_1^2 - 3\al_1\al_{11})v_{42}\\
          & + (4\al_1^2\al_8 - \al_1^2\al_{12} - 3\al_1^2\al_2 - 6\al_1\al_{11}\al_2 
+ 3\al_1\al_{11}\al_8 + \al_1\al_{12}\al_7 + 2\al_1\al_{13}\al_3 \\
          & - \al_1\al_{13}\al_9 + 11\al_1\al_2\al_7 - 7\al_1\al_7\al_8 
+ \al_{11}\al_{13}\al_3 + 6\al_{11}\al_2\al_7 + \al_{13}\al_2^2\\
          & - 2\al_{13}\al_2\al_8 - 3\al_{13}\al_3\al_7 - 8\al_2\al_7^2)v_{58},\\
x_1.v_{13} & = v_{20},\\
x_1.v_{14} & = v_{21} + v_{22},\\
x_1.v_{15} & = v_{22}, \\
x_1.v_{16} & = 2v_{23} + v_{25} - \al_1v_{31} + \al_1\al_{11}v_{42} + (\al_1\al_{12} 
+ \al_{11}\al_2 - \al_{13}\al_3)v_{58},\\
x_1.v_{17} & = 2v_{23} + v_{24},\\
x_1.v_{18} & = v_{26} + 3v_{27} - 3\al_1v_{34} + (2\al_1^2\al_{11} - \al_1^2\al_7 - 2\al_1\al_{11}\al_7 
- 2\al_1\al_{13}\al_2 \\
          & + \al_1\al_{13}\al_8 + \al_1\al_7^2 + 2\al_{13}\al_2\al_7)v_{58},\\
x_1.v_{19} & = 5v_{28} - 10\al_1v_{37} + \al_1\al_{13}(4\al_1\al_7  - \al_1^2 - 3\al_1\al_{11})v_{58},\\
x_1.v_{20} & = v_{30}, \\
x_1.v_{21} & = 2v_{31} - \al_{11}v_{42} - \al_{12}v_{58},\\
x_1.v_{22} & = v_{31},\\
x_1.v_{23} & = v_{32} + v_{33} + v_{34},\\
x_1.v_{24} & = 3v_{33} + (\al_7^2 - 2\al_{11}\al_7 + \al_{13}\al_8)v_{58},\\
x_1.v_{25} & = 2v_{34} - \al_1v_{44} + \al_{13}\al_2v_{58},\\
x_1.v_{26} & = 3v_{35} + v_{37} - 3\al_1v_{45} + \al_1\al_{13}(\al_1 - \al_7)v_{58},\\
x_1.v_{27} & = 2v_{35} + 2v_{36} - \al_1v_{46} - \al_1\al_{13}\al_7v_{58},\\
x_1.v_{28} & = v_{38} + 4v_{39} - 6\al_1v_{50},\\
x_1.v_{29} & = 6v_{40} - 15\al_1v_{53},\\
x_1.v_{30} & = v_{42},\\
x_1.v_{31} & = v_{44},\\
x_1.v_{32} & = v_{43} + 2v_{45},\\
x_1.v_{33} & = 2v_{43} + v_{46} - \al_{13}\al_7v_{58},\\
x_1.v_{34} & = v_{45} + v_{46},\\
x_1.v_{35} & = v_{47} + 2v_{48} + v_{50},
\end{align*}

\begin{align*}
x_1.v_{36} & = 3v_{48} + v_{49},\\
x_1.v_{37} & = 3v_{50},\\
x_1.v_{38} & = 4v_{51} + v_{53},\\
x_1.v_{39} & = 2v_{51} + 3v_{52},\\
x_1.v_{40} & = v_{54} + 5v_{55},\\
x_1.v_{41} & = 7v_{56}, \\
x_1.v_{42} & = v_{58}, \\
x_1.v_{43} & = 0,\\
x_1.v_{44} & =- \al_{13}v_{58},\\ 
x_1.v_{45} & = \cdots = x_1.v_{58}=0. 
\end{align*}

\begin{align*}
x_2.v_{1} & = v_{2}, \quad x_2.v_{2} =v_{4}, \quad x_2.v_{3} =v_6, \quad  x_2.v_{4} =v_{7}, 
\quad x_2.v_5 = v_9,  \\
x_2.v_{6} & = v_{11}, \quad x_2.v_{7} =v_{12}, \quad x_2.v_{8} =v_{15}, \quad  x_2.v_{9} =v_{16}, 
\quad x_2.v_{10} = v_{17},  \\
x_2.v_{11} & = v_{18}, \quad x_2.v_{12} =v_{19}, \quad x_2.v_{13} =0, \quad  x_2.v_{14} =v_{23}, 
\quad x_2.v_{15} = v_{25},  \\
x_2.v_{16} & = v_{26}, \quad x_2.v_{17} =v_{27}, \quad x_2.v_{18} =v_{28}, \quad  x_2.v_{19} =v_{29}, 
\quad x_2.v_{20} = 0,  \\
x_2.v_{21} & = v_{32}, \quad x_2.v_{22} =v_{34}, \quad x_2.v_{23} =v_{35}, \quad  x_2.v_{24} =v_{36}, 
\quad x_2.v_{25} = v_{37},  \\
x_2.v_{26} & = v_{38}, \quad x_2.v_{27} =v_{39}, \quad x_2.v_{28} =v_{40}, \quad  x_2.v_{29} =v_{41}, 
\quad x_2.v_{30} = 0,  \\
x_2.v_{31} & = v_{45}, \quad x_2.v_{32} =v_{47}, \quad x_2.v_{33} =v_{48}, \quad  x_2.v_{34} =v_{50}, 
\quad x_2.v_{35} = v_{51},  \\
x_2.v_{36} & = v_{52}, \quad x_2.v_{37} =v_{53}, \quad x_2.v_{38} =v_{54}, \quad  x_2.v_{39} =v_{55}, 
\quad x_2.v_{40} = v_{56},  \\
x_2.v_{41} & = v_{57}, \quad x_2.v_{42} = \cdots = x_2.v_{58}=0.
\end{align*}
\end{proof}

\begin{cor}\label{quoti}
There is a general faithful $\Lf$-module $V_{20}=V_{20}(\al_1,\ldots ,\al_{13})$ of 
dimension $20$. 
\end{cor}

\begin{proof}
We apply the algorithm {\sf Quotient} from \cite{BEG} to the module $V_{58}$.
This works as follows. The space of invariants is given by 
\[
V_{58}^{\Lf}=\langle \al_{13}v_{42}+v_{44},v_{43},v_{45},\ldots ,v_{58} \rangle,
\]
with $\dim (V_{58}^{\Lf})=16$ for all parameters $\al_1,\ldots ,\al_{13}$.
We choose a complement $U$ of  $Z(\Lf)=\langle v_{58} \rangle$ in $V_{58}^{\Lf}$
by taking the above basis for $V_{58}^{\Lf}$ except for $v_{58}$. Then $U$ is
a submodule such that the quotient $V_{43}=V_{58}/U$ is a faithful module of dimension $43$.
For the quotient, we may write the following relations 
\begin{align*}
v_{43} & = 0, \\
v_{44} & = -\al_{13}v_{42},\\
v_{45} & = \cdots = v_{57}=0.
\end{align*}
In other words, we may view $v_1,\ldots ,v_{42},v_{58}$ as a basis of
$V_{43}$. Now we repeat this procedure. We have
\[
V_{43}^{\Lf}=\langle \al_{13}v_{30}+v_{31},v_{32},v_{33}+\al_7\al_{13}v_{42},v_{34},\ldots ,v_{41},v_{58} \rangle,
\]
with $\dim (V_{43}^{\Lf})=12$ for all parameters $\al_1,\ldots ,\al_{13}$. We choose $U$ from $
V_{43}^{\Lf}$ by omitting $v_{58}$, and obtain a faithful quotient $V_{32}=V_{43}/U$ of dimension $32$.
We can take the following quotient relations 
\begin{align*}
v_{31} & =  -\al_{13}v_{30}, \\
v_{32} & = 0, \\
v_{33} & = -\al_7\al_{13}v_{42},\\
v_{34} & = \cdots = v_{41}=0.
\end{align*}
In the next step we obtain $\dim (V_{32}^{\Lf})=10$ for all parameters $\al_1,\ldots ,\al_{13}$.
Choosing a complement $U$ as above we obtain a faithful module $V_{23}=V_{32}/U$ of dimension $23$,
where the relations are given by
\begin{align*}
v_{21} & =  -2\al_{13}v_{20}-\al_{11}v_{30}-\al_{12}v_{42}, \\
v_{22} & = -\al_{13}v_{20},\\
\vdots  & = \vdots \\
v_{29} & = 0.
\end{align*}
The dimension of the space of invariants $V_{23}^{\Lf}$ however does depend on the parameters.
It can be of dimension $5$,$6$ or $7$, depending on certain case distinctions.
Without case distinction we can still choose some subspace $U$ of invariants not containing $v_{58}$,
which need not be a maximal with this property. This way we arrive at a faithful quotient $V_{20}$ of
dimension $20$. If we continue with case distinctions we obtain many different faithful quotients $V$
of dimensions $10\le \dim (V)\le 18$. The quotient algorithm stops if the space of invariants
is $1$-dimensional, spanned by $v_{58}$. Then there is no faithful quotient of lower
dimension.
\end{proof}

\begin{rem}
Note that the choice of the complements $U$ in the quotient algorithm is not unique. For our choice 
we obtained faithful modules of dimensions $58$, $43$, $32$ and $23$. In general, the dimensions 
might depend on $U$. However, taking quotients by invariants is no restriction. In fact, the
following result is easy to show: let $\Ln$ be a nilpotent Lie algebra, and $V$ be a nilpotent 
$\Ln$-module. Then every faithful quotient of $V$ can be obtained by taking successive quotients 
by invariants.
\end{rem}

\begin{ex}\label{4.11}
Consider the Lie algebra  $\Lf=\Lf(\al_1,\ldots ,\al_{13})$
with 
\[
(\al_1,\ldots ,\al_{13})=(1,0,0,0,0,0,-1,1,0,0,3,-16,1).
\]
We have $\mu(\Lf)\ge 12$, and $\Lf$ admits no affine structure, see \cite{BU5}. 
The above algorithm yields a faithful quotient of $V_{58}$ of dimension $18$.
Hence we have $\mu(\Lf)\le 18$, and this is up to now the best known estimate.
\end{ex}

Note that the above Lie algebra has minimal $\be$-invariant, namely $\be(\Lf)=5$.
The Betti numbers are given by 
$(b_0,\ldots , b_{10})=(1,2,3,5,6,6,6,5,3,2,1)$. \\[0.2cm]
We come back to finding as good as possible estimates on $\mu(\Lf)$ for all filiform Lie algebras 
$\Lf=\Lf(\al_1,\ldots ,\al_{13})$ of dimension $10$. Therefore we need to consider different 
choices of admissible parameters, which give well-defined classes of filiform Lie algebras.
The cases are as follows: \\[0.2cm]
{\it Case $1$:} $2\al_1+\al_7=0$. \\[0.1cm]
{\it Case $2$:} $2\al_1+\al_7\neq 0$. \\[0.1cm]
\ph {\it Case $2a$:} $\al_{13}\neq 0$, $\al_7^2=\al_1^2\neq 0$. \\[0.1cm]
\ph \ph {\it Case $2a1$:} $\al_7=\al_1$.  \\[0.1cm]
\ph \ph {\it Case $2a2$:} $\al_7= -\al_1$. \\[0.1cm]
\ph \ph \ph {\it Case $2a2a$:} $3\al_2+\al_8=0$. \\[0.1cm]
\ph \ph \ph {\it Case $2a2b$:} $3\al_2+\al_8 \neq 0$. \\[0.1cm]
\ph {\it Case $2b$:} $\al_{13}=0$. \\[0.1cm]
\ph \ph {\it Case $2b1$:} $\al_7^2\neq \al_1^2$. \\[0.1cm]
\ph \ph {\it Case $2b2$:} $\al_7^2=\al_1^2$. \\[0.1cm]
\ph \ph \ph {\it Case $2b2a$:} $\al_7=\al_1$. \\[0.1cm]
\ph \ph \ph {\it Case $2b2b$:} $\al_7=-\al_1$. \\[0.1cm]
\ph \ph \ph \ph {\it Case $2b2b1$:} $3\al_2+\al_8=0$. \\[0.1cm]
\ph \ph \ph \ph {\it Case $2b2b2$:} $3\al_2+\al_8 \neq 0$. 

\begin{lem}
All above conditions are isomorphism invariants. In particular, algebras of different cases are
non-isomorphic.
\end{lem}

\begin{proof}
Using the $\be$-invariant we have
\begin{align*}
\al_1=0 \; & \Leftrightarrow \; \be(\Lf/\Lf^5) = 4,\\
\al_7=0 \; & \Leftrightarrow \; \be(\Lf^2/\Lf^7)=5,\\
\al_{11}=0 \; & \Leftrightarrow \; \be(\Lf^3/\Lf^9)=6,\\
\al_{13}=0 \; &  \Leftrightarrow \; \be(\Lf^4/\Lf^{11})=6, \\
\al_7=\al_1 \; & \Leftrightarrow \; \be (\Lf/\Lf^2\ltimes \Lf^4/\Lf^7)=4.
\end{align*}
The Lie algebras of case $1$ satisfy $2\al_1+\al_7=0$, which is equivalent to
$\al_1=\al_7=0$. The above table shows that these conditions are isomorphism invariants.
Hence the Lie algebras of case $1$ and case $2$ are well-defined. The same applies
to case $2a$ and case $2b$, because $\al_{13}\neq 0$ and $\al_{13}= 0$ are isomorphism invariants.
Recall that $\al_{13}\neq 0$ implies $\al_7^2=\al_1^2$.
The claim is also clear for the cases $2a1$, $2a2$.
Note that $\al_7=\al_1\neq 0$ is also equivalent to the conditions $\be(\Lf^2/\Lf^7)\neq 5$
and $[\Lf^2,\Lf^5]=\Lf^9$. As we will see in proposition $\ref{4.17}$, the Lie algebras
of case $2b1$ are well-defined. Finally, for the cases with $\al_7=-\al_1\neq 0$ the condition
$3\al_2+\al_8=0$ is equivalent to the fact, that the Lie algebra $\Lf/\Lf^8$ admits an invertible
derivation. Hence this condition is also an isomorphism invariant.
\end{proof}

For each case we have a result on $\mu(\Lf)$. Let us start with the first case.

\begin{prop}
Let $\Lf=\Lf(\al_1,\ldots ,\al_{13})$ be a filiform nilpotent Lie algebra of dimension $10$ satisfying
$2\al_1+\al_7=0$. Then $\mu(\Lf)=10$.
\end{prop}

\begin{proof}
The parameters are admissible iff $\al_1 =\al_7 = 0$ and $\al_{11}(\al_2+\al_8) = 0$.
To construct a module for $\Lf$ we need to find two operators
$L(x_1)$ and $L(x_2)$, which define $L(x_i):=[L(x_1),L(x_{i-1})]$ for $i\ge 3$, so that the
conditions $L([x_i,x_j)]=[L(x_i),L(x_j)]$ are satisfied for all  $i,j\ge 1$.
This module is faithful if and only if $L(x_{10})$ is nonzero.
It is easy to see that we can always find such operators, by taking $L(x_1)=\ad(x_1)$ and 
$L(x_2)$ some $10\times 10$ lower-triangular matrix. However, the construction depends on
different cases, such as $\al_{13}\neq 0$, or $\al_{13}=0$ with
$\al_{11}\neq 0, \al_2\neq 0$, with $\al_{11}\neq 0, \al_2= 0$, or with $\al_{11}=0$. 
For more details see \cite{BU5}.
\end{proof}

For case $2a$ we have the following results, see \cite{BU5}:

\begin{prop}
Let $\Lf=\Lf(\al_1,\ldots ,\al_{13})$ be a filiform nilpotent Lie algebra of dimension $10$ satisfying
$2\al_1+\al_7\neq 0$, $\al_{13}\neq 0$ and $\al_7=\al_1$. Then $\mu(\Lf)\le 11$.
\end{prop}

\begin{prop}
Let $\Lf=\Lf(\al_1,\ldots ,\al_{13})$ be a filiform nilpotent Lie algebra of dimension $10$ satisfying
$2\al_1+\al_7\neq 0$, $\al_{13}\neq 0$ and $\al_7=-\al_1$. Then $\mu(\Lf)\le 11$ if and only if 
$3\al_2+\al_8=0$. Otherwise we have $\mu(\Lf)\le 18$. 
\end{prop}

In this case the module $V_{58}$ from proposition $\ref{4.10}$ always has a faithful
quotient of dimension $18$. This can be seen by applying the quotient algorithm as
in corollary $\ref{quoti}$. For $3\al_2+\al_8\neq 0$ this is the best bound known so far.
The example given in $\ref{4.11}$ belongs to this class. \\
For case $2b$ we have the following results, see \cite{BU5} and \cite{BU13}:

\begin{prop}\label{4.17}
Let $\Lf=\Lf(\al_1,\ldots ,\al_{13})$ be a filiform nilpotent Lie algebra of dimension $10$ satisfying
$2\al_1+\al_7\neq 0$. Then  $\Lf$ admits a central extension 
$0 \ra Z(\Lh)\ra \Lh\ra \Lf\ra 0$ by some filiform nilpotent Lie algebra $\Lh$ 
if and only if $\al_{13}= 0$ and $\al_1^2\neq \al_7^2$, in which case we have  $\mu(\Lf)=10$.
\end{prop}

\begin{prop}
Let $\Lf=\Lf(\al_1,\ldots ,\al_{13})$ be a filiform nilpotent Lie algebra of dimension $10$ satisfying
$2\al_1+\al_7\neq 0$, $\al_{13}= 0$ and $\al_7=\al_1$. Then  $\mu(\Lf)\le 11$.
\end{prop}

\begin{prop}
Let $\Lf=\Lf(\al_1,\ldots ,\al_{13})$ be a filiform nilpotent Lie algebra of dimension $10$ satisfying
$2\al_1+\al_7\neq 0$, $\al_{13}= 0$ and $\al_7=-\al_1$. Then $\mu(\Lf)\le 11$ if and only if 
$3\al_2+\al_8=0$. Otherwise we have $\mu(\Lf)\le 15$. 
\end{prop}

Here we use proposition $\ref{4.10}$ for the subcase $3\al_2+\al_8 \neq 0$. Then the module
$V_{58}$ has a faithful quotient of dimension $14$. In fact, for some cases, it even has
a faithful quotient of dimension $12$, $13$ or $14$.

%%%%%%%%%%%%%%%%%%%%%%%%%%%%%%%%%%%%%%%%%%%%%%%%%%%%%%%%%%%%%%%%%%%%%%%%%%%%%%


\begin{thebibliography}{99}


\bibitem{BIR} G. Birkhoff: {\it Representability of Lie algebras and 
Lie groups by matrices}. Annals of Mathematics \textbf{38} (1937), 526--532.

\bibitem{BU5} D. Burde: {\it Affine structures on nilmanifolds}. Int.\ J.\
of Math.\ \textbf{7} (1996), no. 5, 599--616.

\bibitem{BU7} D. Burde: {\it A refinement of Ado's Theorem}. 
Archiv Math. \textbf{70} (1998), 118--127.

\bibitem{BU13} D. Burde: {\it Affine cohomology classes for filiform 
Lie algebras}. Contemporary Mathematics \textbf{262} (2000), 159--170.

\bibitem{BUM} D. Burde, W. Moens: {\it Minimal faithful representations of
reductive Lie algebras}. Archiv der Mathematik \textbf{89} (2007), no. 6, 513--523.

\bibitem{BEG} D. Burde, B. Eick, W. de Graaf: {\it Computing faithful representations
for nilpotent Lie algebras}. J.\ Algebra \textbf{22} (2009), no. 3, 602--612.

\bibitem{CAR} L. Cagliero, N. Rojas: {\it Faithful representations of minimal dimension of 
current Heisenberg Lie algebras}. Internat. J. Math.  \textbf{20}  (2009), no. 11, 1347--1362. 

\bibitem{GRA} W. A. de Graaf: {\it Constructing faithful matrix 
representations of Lie algebras}. Proceedings of the 1997 International 
Symposium on Symbolic and Algebraic Computation, ISSAC'97, (1997), 54--59. 
ACM Press New York.

\bibitem{MIL} J. Milnor, {\it On fundamental groups of complete affinely
flat manifolds}, Adv.\ Math.\
\textbf{25} (1977), 178-187.

\bibitem{NER} Y. Neretin: {\it A construction of finite-dimensional faithful representation of 
Lie algebra}.  Rend. Circ. Mat. Palermo (2) Suppl.  No. \textbf{71}  (2003), 159--161.   

\bibitem{VER} M. Vergne: {\it Cohomologie des alg\`ebres de Lie nilpotentes.
Application \`a l'\'etude de la vari\'et\'e des alg\`ebres de Lie
nilpotentes}. Bull.\ Soc.\ Math.\ France \textbf{98} (1970), 81-116.



\end{thebibliography}
\end{document}